\documentclass[11pt,reqno]{amsart}

\usepackage{amsmath, amsfonts, amsthm, amssymb, color}

\newtheorem{Theorem}{Theorem}[section]
\newtheorem{Lemma}{Lemma}[section]
\newtheorem{Proposition}{Proposition}[section]

\theoremstyle{definition}
\newtheorem{Definition}{Definition}[section]

\theoremstyle{remark}
\newtheorem{Remark}{Remark}[section]

\numberwithin{equation}{section}

\renewcommand{\u}{{\bf u}}

\newcommand{\R}{{\mathbb R}}
\newcommand{\Dv}{{\rm div}}

\newcommand{\x}{{\bf x}}

\newcommand{\m}{{\bf m}}

\def\w{w}

\def\f{\frac}
\renewcommand{\O}{\Omega}

\def\D{\Delta }

\def\hf1{^\f{1}{1-\xi^2}}

\def\be{\begin{equation}}
\def\en{\end{equation}}
\def\bs{\begin{split}}
\def\es{\end{split}}

\renewcommand{\v}{{\bf v}}

\author{Dehua Wang}
\address{Department of Mathematics, University of Pittsburgh,
                           Pittsburgh, PA 15260.}
\email{dwang@math.pitt.edu}
\author{Cheng Yu}
\address{Department of Mathematics, University of Pittsburgh,
                           Pittsburgh, PA 15260.}
\email{chy39@pitt.edu}

\title
[Navier-Stokes-Vlasov equations]
{Global weak solutions to the inhomogeneous Navier-Stokes-Vlasov equations}


\date{\today}

\begin{document}

\begin{abstract}
 A fluid-particle system  of the incompressible inhomogeneous Navier-Stokes equations and Vlasov equation in the three dimensional space is considered in this paper. The coupling arises from the drag force in the fluid equations and the acceleration in the Vlasov equation. An initial-boundary value problem is studied in a bounded domain
with large data. The existence of global weak solutions is
established through an approximation scheme, energy estimates, and weak convergence.
\end{abstract}

\maketitle

\section{Introduction}



We are concerned with the Navier-Stokes-Vlasov equations for particles dispersed in a density-dependent incompressible viscous fluid:
\begin{equation}\label{1}
\rho_{t}+\Dv{(\rho\u)}=0,
\end{equation}
\begin{equation}\label{2}
(\rho\u)_{t}+\Dv(\rho\u\otimes \u)+\nabla p-\mu\D\u=-\int_{\R^3} {m_{p}Ff}d\v,
\end{equation}
\begin{equation}\label{3}
\Dv \u=0,
\end{equation}
\begin{equation}\label{4}
f_{t}+\v\cdot\nabla_{x}f+\Dv_{\v}{(Ff)}=0,
\end{equation}
for $(x,\v,t)$ in $\O\times\R^3\times(0,\infty)$,
where $\O\subset\R^3$ is a bounded  domain,
$\rho$ is the density of the fluid,
$\u$ is the velocity of the fluid, $p$ is the pressure, $\mu$ is the kinematic viscosity of the fluid.
The density distribution function $f(t,x,\v)$ of particles
depends on the time $t\in[0,T]$, the physical position $x\in\O$ and the velocity of particle $\v\in\R^3$. 
In \eqref{2}, $m_{p}$ is the mass of the particle and $F$ is the drag force.
The interaction of the fluid and particles is through the
drag force exerted by the fluid onto the particles.  The drag force $F$
typically depends on the relative velocity $\u-\v$ and on the density of
fluid $\rho$ (e.g. \cite{MV2}),  such as
\begin{equation}
\label{5}
F=F_0\rho(\u-\v),
\end{equation}
where $F_0$ is a positive constant.
Without loss of generality we take $\mu=F_0=m_p=1$ throughout the paper.
The fluid-particle system \eqref{1}-\eqref{5} arises in many applications such as
sprays, aerosols, and more general  two phase flows where
one phase (disperse) can be considered as a suspension of particles onto the other
one (dense) regarded as a fluid.
This system \eqref{1}-\eqref{5} or its variants have been used in sedimentation
of solid grain by external forces, for fuel-droplets in combustion theory (such as in the study of engines),  chemical engineering,  bio-sprays in medicine, waste water treatment, and pollutants in the air.
We refer the readers to \cite{BD,CP,D,DR,DV,GHMZ,O,RM,W} for more physical background, applications and discussions of  the fluid-particle systems.

The aim of this paper is to establish the global existence of weak solutions to
the initial-boundary value problem for the system \eqref{1}-\eqref{5} subject to the following initial data:
\begin{equation}\label{6}
\rho|_{t=0}=\rho_{0}(\x)\ge 0,\;\;  \x \in  \O,
\end{equation}
\begin{equation}\label{7}
(\rho\u)|_{t=0}=\m_{0}(x),\;\; \x\in \O,
\end{equation}
\begin{equation}\label{8}
f|_{t=0}=f_{0}(x,\v),\;\; \x\in \O,\quad\v\in \R^{N},
\end{equation}
and the following boundary conditions:
\begin{equation}
\label{boundary condition}
\begin{split}
&\u(t,x)=0\;\;\text{ on } \partial\O,\\
&f(t,x,\v)=f(t,x,\v^{*})\;\; \text {for } x\in\partial\O,\;\v\cdot\nu(x)<0,
\end{split}
\end{equation}
where $$\v^{*}=\v-2(\v\cdot \nu(x))\nu(x)$$ is the specular velocity,
and $\nu(x)$ is the outward normal vector to $\O$. The same as in \cite{H}, we assume that the particles are reflected by the boundary following the specular reflection laws.
When the drag force is assumed independent of density in \eqref{5},
the hydrodynamic limits and the global existence of weak solutions to the Navier-Stokes and Vlasov-Fokker-Planck equations were studied in \cite{GJV,GJV2,MV,MV2}.  When the drag force  depends on the density  as in \eqref{5}, a relaxation of the kinetic regime toward a hydrodynamic regime with
velocity $\u$ on the vacuum $\{\rho=0\}$ can not be excepted.
It is hard to establish a priori lower bounds on the density from the mathematics view point. The objective of this paper is to
establish the existence of global weak solutions to the initial-value problem
\eqref{1}-\eqref{boundary condition}  with large data in  certain functional spaces.
In general, the analysis of fluid-particle system is challenging since the density distribution function $f$ depends on more variables than the fluid density $\rho$ and velocity $\u$.
The existence of global weak solutions to the Stokes-Vlasov equations in a bounded domain was studied in \cite{H}. In \cite{BDGM}  the convection term was included and   the incompressible Navier-Stokes-Vlasov equations were considered in periodic domains. A similar system with  thermal diffusion acting on the particles, that is, the incompressible Navier-Stokes equations coupled with the Vlasov-Fokker-Planck equations, has been studied in \cite{GHMZ}, where the authors established the existence of classical solutions with small data.
Recently, the existences of global solutions to the incompressible Navier-Stokes-Vlasov equations in a bounded domain or in the whole space were obtained in \cite{YU1,YU2}. Also there are a lot of works on hydrodynamic limits, we refer the reader to \cite{CP,GJV,GJV2,MV2} (and the references therein) where  some scaling  and convergence methods such as the compactness and relative entropy method were applied to investigate the hydrodynamic limits. A key idea in \cite{GJV,GJV2} is to control  the dissipation rate of a certain free energy associated with the whole space.
The global existence of weak solutions to the compressible Navier-Stokes equations coupled to Vlasov-Fokker-Planck equations was established in \cite{MV}. The coupled system \eqref{1}-\eqref{5} has extra difficulties due to the appearance of density   in the interactions and in the Vlasov equation as well as the lack of diffusion in the Vlasov equation. We note that the local classical solution to the Euler-Vlasov equations was obtained in \cite{BD} when the drag force $F$ is assumed to be in the form of \eqref{5}.

When $f$ does not appear, the system \eqref{1}-\eqref{4} is reduced to the density-dependent Navier-Stokes equations. We refer readers to Lions \cite{L} for the compactness, the existence of global weak solutions and more background discussions on the density-dependent Navier-Stokes equations. For the Navier-Stokes equations, it is necessary for the external forces to be in the functional space $L^{2}(\Omega\times(0,T))$ in order to obtain the global existence. However, for the system \eqref{1}-\eqref{5} the term $-\int(\u-\v)\rho f d\v$ does not have enough regularity.  To overcome this difficulty, we  decompose the term into two components:
\begin{equation}
\label{decompose}
-\int_{\R^3}(\u-\v)\rho f d\v=-\rho\u\int_{\R^3}f\,d\v+\rho\int_{\R^3}\v f\,d\v,
\end{equation}
and we can view $\rho\int_{\R^3}\v f\,d\v$ as the external force of the Navier-Stokes equations. As we shall see later on, the work of internal forces $\rho\u\int_{\R^3}f\,d\v$ appears on the left side of Navier-Stokes equations and has to be nonnegative to keep the energy inequality. Meanwhile, we
introduce a  regularization function $R_{\delta}$ (see Section 3 for the definition) as in \cite{H} to construct an approximation of \eqref{decompose}:
\begin{equation*}
-R_{\delta}\int_{\R^3}(\u-\v)\rho f d\v=-\rho R_{\delta}(\int_{\R^3}f\,d\v)\u+\rho R_{\delta}\int_{\R^3}\v f\,d\v.
\end{equation*}
To keep a similar energy inequality for the approximation scheme, we need to add the regularized acceleration in the Vlasov equation too.
Then we see that the external force term is in $L^{2}(0,T;\O)$ and the internal forces is  finite when $\delta$ is fixed, hence we can solve the regularized Navier-Stokes equations.  Indeed, we can obtain the smooth solution of the regularized Navier-Stokes equations when the initial data is good enough.  The uniqueness and existence of the Vlasov equation can be obtained when $(\rho,\u)$ is smooth, see \cite{DL, H}. The next step is to pass the limit to recover the original system from the approximation scheme. We shall see that the $L^p$ regularity of velocity averages \cite{DLM} and fine compactness of the system guarantee
the existence of global weak solutions to \eqref{1}-\eqref{boundary condition}.

We organize the rest of the paper as follows. In Section 2, we deduce a prior estimates from \eqref{1}-\eqref{5}, give the definition of weak solutions, and also state our main results. In Section 3, we construct an approximation scheme to \eqref{1}-\eqref{5}, establish its global existence, use the uniform estimates and $L^p$ average velocity lemma to recover the original system.

\bigskip

\section{A Priori Estimates and Main Results}

In this section, we shall derive some fundamental a priori estimates and then state our main results. These estimates will play an important role in the compactness analysis later since they will allow us to deduce the global existence upon passing to the limit in the regularized  approximation scheme.
We shall develop these  a priori  estimates in the three-dimensional space, but they all hold in the two-dimensional space.

First, roughly speaking, \eqref{1} and the incompressibility condition mean that the density $\rho(t,x)$ is independent of time $t$. In fact, we take any function $\beta\in C^{1}([0,\infty);\O)$, multiply \eqref{1} by $\beta'(\rho)$, use the incompressibility condition, and integration by parts over $\O$, then we have
\begin{equation*}
\frac{d}{dt}\int_{\O}\beta(\rho)dx=0.
\end{equation*}
Applying the maximum principle to the transport equations \eqref{1} and \eqref{3}, one deduces that
\begin{equation*}
\|\rho\|_{L^{\infty}}\leq\|\rho_{0}\|_{L^{\infty}},
\end{equation*}
and also $\rho\geq 0$, so we have
\begin{equation}
\label{2.1++}
0 \leq \rho(t,x)\leq\|\rho_{0}\|_{L^{\infty}}
\end{equation}
for almost every $t$.

We now multiply \eqref{2} by $\u$ and integrate over $\O$, and use \eqref{1}, \eqref{3}, and \eqref{5} to deduce that
\begin{equation}\label{2.1}
\frac{d}{dt}\int_{\O}\rho|\u|^{2}\;dx+2\int_{\O}|\nabla\u|^{2}\;dx
=-2\int_{\O}\int_{\R^{3}}\rho f(\u-\v)\cdot\u \;d\v\; dx.
\end{equation}
On the other hand, we multiply the Vlasov equation \eqref{4} by $\frac{|\v|^{2}}{2}$, integrate over $\O\times\R^{3}$, and use integration by parts to obtain

\begin{equation}\begin{split}\label{2.2}
&\frac{d}{dt}\int_{\O}\int_{\R^3}f|\v|^{2}\;d\v dx-\int_{\partial\O\times\R^3}(\v\cdot \nu)\frac{|\v|^2}{2}\rho f \,d\v\,dx
\\&=-2\int_{\O}\int_{\R^3}\rho f|\u-\v|^2 \;d\v dx
+2\int_{\O}\int_{\R^{3}}\rho f(\u-\v)\u\; d\v dx.
\end{split}\end{equation}

We can rewrite $\v^{*}$ as follows
\begin{equation*}
\v^{*}=R\v,\;\;\;\text{ where } R=I-2\nu\nu^{T},
\end{equation*}
where $\nu$ is the outward unit normal vector to $\partial \O$ for all $x\in\partial\O$.
By direct computation, one obtains that $|\v^*|^2\leq |\v|^2.$
Thus, we can treat the boundary term in \eqref{2.2} as follows:
\begin{equation}
\label{boundaryterm}
\begin{split}
&\int_{\partial\O\times\R^3}(\v\cdot\nu)\frac{|\v|^2}{2}\rho f \,d\v\,dx=\int_{\v\cdot\nu>0}(\v\cdot\nu)\frac{|\v|^2}{2}\rho f \,d\v\,dx+\int_{\v\cdot\nu<0}(\v\cdot\nu)\frac{|\v|^2}{2}\rho f \,d\v\,dx
\\&\leq\int_{\v\cdot\nu>0}(\v\cdot\nu)\frac{|\v|^2}{2}\rho f \,d\v\,dx-\int_{\v^*\cdot\nu>0}(\v^*\cdot\nu)\frac{|\v^*|^2}{2}\rho f \,d\v^*\,dx
=0.
\end{split}
\end{equation}

It is easy to get \begin{equation*}
\frac{d}{dt}\int_{\O}\int_{\R^3}f\; d\v dx-\int_{\partial\O\times\R^3}(\v\cdot\nu) \rho f \,d\v\,dx=0,
\end{equation*}
we treat the boundary term as follows
\begin{equation*}
\begin{split}
&\int_{\partial\O\times\R^3}(\v\cdot\nu) \rho f \,d\v\,dx
=\int_{\v\cdot\nu>0}(\v\cdot\nu) \rho f \,d\v\,dx+\int_{\v\cdot\nu<0}(\v\cdot\nu) \rho f \,d\v\,dx
\\&=\int_{\v\cdot\nu>0}(\v\cdot\nu) \rho f \,d\v\,dx-\int_{\v^*\cdot\nu>0}(\v^*\cdot\nu) \rho f(t,x,\v^*) \,d\v^*\,dx=0,
\end{split}
\end{equation*}
which implies that the conservation of mass:
\begin{equation} \label{CM}
\frac{d}{dt}\int_{\O}\int_{\R^3}f\; d\v dx=0.
\end{equation}
This together with \eqref{2.1}, \eqref{2.2} and \eqref{boundaryterm},
we obtain the following energy equality for the system
\eqref{1}-\eqref{5}:
\begin{equation}
\begin{split}
\label{2.3}
&\frac{d}{dt}\int_{\O}\rho|\u|^{2}\;dx+\frac{d}{dt}\int_{\O}\int_{\R^3}f(1+|\v|^{2})\;d\v
dx+2\int_{\O}\int_{\R^3}\rho f|\u-\v|^2\; d\v
dx\\&+2\int_{\O}|\nabla\u|^{2}\;dx
\leq0.
\end{split}\end{equation}

Integrating \eqref{2.3} with respect to $t$, we obtain for all $t$,
\begin{equation}\begin{split}\label{2.2++}
&\int_{\O}\rho|\u|^{2}\;dx+\int_{\O}\int_{\R^3} f(1+|\v|^{2})\;d\v
dx+2\int_{0}^{t}\int_{\O}\int_{\R^3}\rho f|\u-\v|^2 \;d\v
dxdt\\&+2\int_{0}^{t}\int_{\O}|\nabla\u|^{2}\;dxdt
\\&\leq \int_{\O}\frac{|\m_{0}|^2}{\rho_{0}}\;dx+\int_{\O}\int_{\R^3}f_{0}(1+|\v|^{2})\;d\v dx.
\end{split}
\end{equation}
By \eqref{2.3}, it is easy to find that the global energy is non-increasing with respect to $t$:
\begin{equation*}
\frac{d}{dt}\left(\int_{\O}\rho|\u|^{2}\;dx+\int_{\O}\int_{\R^3}f(1+|\v|^{2})\;d\v dx\right)\le 0.
\end{equation*}
Assume
 \begin{equation*}
\int_{\O}\frac{|\m_{0}|^2}{\rho_{0}}\;dx+\int_{\O}\int_{\R^3}f_{0}(1+|\v|^{2})\;d\v dx< \infty,
\end{equation*}
then
\begin{equation*}
\int_0^t \int_{\O}\int_{\R^3} \rho f|\u-\v|^2 \;d\v dxdt\leq C,
\end{equation*}
and
\begin{equation}
\label{2.3++}
\|\nabla\u\|_{L^{2}((0,T)\times\O)}\leq C,
\end{equation}
\begin{equation}
\label{2.4+}
\sup_{0\leq t\leq T}\|\rho|\u|^{2}\|_{L^{1}(\O)}\leq C,
\end{equation}
for any given $T>0$ and some generic positive constant $C$. 
Moreover, by the Poincar\'e inequality we obtain
\begin{equation}
\label{2.5+}
\|\u\|_{L^{2}(0,T;H^1_0(\O))}\leq C.
\end{equation}

The maximum principle applied to \eqref{4} implies that
\begin{equation}
\label{2.6+}
\|f\|_{L^{\infty}}\leq C\|f_{0}\|_{L^{\infty}}
\end{equation}
for all $t\in [0,T]$.   Moreover,  $f_{0}\geq 0$ implies $f\geq 0$ for almost every  $(t,x,\v).$
Then, by the conservation of mass \eqref{CM}  and \eqref{2.6+}, one has   the following estimate:
\begin{equation} \label{EE1}
\begin{split}
&\|f\|_{L^{\infty}((0,T)\times\O\times\R^3)}+\|f\|_{L^{\infty}(0,T;L^{1}(\O\times\R^3))}
\\& \leq C\left(\|f_{0}\|_{L^{\infty}((0,T)\times\O\times\R^3)}+\|f_{0}\|_{L^{\infty}(0,T;L^{1}(\O\times\R^3))}\right).
\end{split}
\end{equation}



Let $\w(t,x)$ be a smooth vector field in $\R^3$ and let $f$ be a solution to the following kinetic equation:
\begin{equation}
\label{2.7+}
\begin{split}
&\partial_{t} f+\v \cdot \nabla_{x} f+\Dv_{\v}((\w-\v)f)=0,
\\& f(0,x,\v)=f_{0}(x,\v),\;\;f(t,x,\v)=f(t,x,\v^{*}) \;\text{ for }x\in\partial\O,\;\; \v\cdot\nu(x)<0
\end{split}
\end{equation}
in $\O\times\R^3.$
 DiPerna-Lions \cite{DL} obtained the existence and uniqueness of solution to \eqref{2.7+} when $\w$ is not smooth. 
Denote the moments of $f$   by
 \begin{equation*}
 \begin{split}
 &m_{k}f(t,x)=\int_{\R^3}f(t,x,\v)|\v|^{k}\;d\v,
 \\&M_{k}f(t)=\int_{\O}\int_{\R^3}f(t,x,\v)|\v|^{k}\;d\v dx,
 \end{split}
 \end{equation*}
 for any $t\in [0,T]$, $x \in\O$, and  integer $k \geq 0$.
 It is clear that
 \begin{equation*}
 M_{k}f(t)=\int_{\O}m_{k}f(t,x)dx.
 \end{equation*}
We first recall the following lemma \cite{H}:
\begin{Lemma}\label{l2+}
Let $\w\in L^{p}(0,T;L^{N+k}(\O))$ with $1\leq p\leq\infty$ and $k\geq1.$
Assume that $f_{0}\in (L^{\infty}\cap L^{1})(\O\times\R^3)$ and $m_{k}f_{0}\in L^{1}(\O\times\R^3)$. Then, the solution $f$ of \eqref{2.7+} should have the following estimates
\begin{equation*}
M_{k}f\leq C\left((M_{k}f_{0})^{1/(N+k)}+(|f_{0}|_{L^{\infty}}+1)\|\w\|_{L^{p}(0,T;L^{N+k}(\O))}\right)^{N+k}
\end{equation*}
for all $0\leq t\leq T$ where the constant $C $  depends only on $T$.
\end{Lemma}

We also recall the average compactness result for the Vlasov equation due to Di Perna-Lions-Meyer \cite{DLM}:
\begin{Lemma}
\label{l3+}
Suppose
\begin{equation*}
\frac{\partial f^{n}}{\partial t}+\v\cdot\nabla_{x}f^n=\Dv_{\v}(F^n f^n)\quad\quad\quad\text{ in } \mathcal{D}'(\O+\R^3\times(0,\infty))
\end{equation*}
where $f^n$ is bounded in $L^{\infty}(0,\infty;L^2(\O\times\R^3))$ and  $f^n(1+|\v|^2)$ is bounded in $ L^{\infty}(0,\infty; L^{1}(\O\times\R^3)$, $\frac{F^n}{1+|\v|}$ is bounded in $L^{\infty}((0,\infty)\times \R^3;L^{2}(\O)).$
Then $\int_{\R^3}f^n\eta(\v)d\v$ is relatively compact in $L^{q}(0,T;L^{p}(\O))$ for $1\leq q <\infty, 1\leq p< 2$ and for $\eta$ such that $\frac{\eta}{(1+|\v|)^\sigma}\in L^1+L^{\infty}, \sigma \in [0,2).$
\end{Lemma}

\begin{Remark}
We shall use this lemma for the Vlasov equation to obtain the compactness of $m_0f$ and $m_1f$, which will allow us to pass the limit when $\varepsilon$ and $\delta$ go to zero in the approximation.
\end{Remark}

In this paper, we assume that
\begin{equation}\begin{split}\label{2.4}
&\rho_{0}\geq 0 \; \text{ almost everywhere in } \O, \;\; \rho_{0}\in L^{\infty}(\O),
\\& \m_{0}\in L^2(\O),\quad \m_{0}=0\; \text{ almost everywhere on } \{\rho_{0}=0\},
\quad \frac{|\m_{0}|^2}{\rho_{0}} \in L^1(\O),
\\& f_{0}\in L^{\infty}(\O\times\R^3),\quad m_{3}f_{0} \in L^{\infty}((0,T)\times \O).
\end{split}\end{equation}

\begin{Definition} \label{D1}
We say that $(\rho,\u,f)$ is a global weak solution to problem \eqref{1}-\eqref{8} if the following conditions are satisfied: for any $T>0$,
\begin{itemize}
\item  $\rho\geq 0, \quad \rho \in L^{\infty}([0,T]\times \O),\quad \rho \in C([0,T];L^{p}(\O)), \quad 1 \leq p <\infty;$
\item $\u \in L^{2}(0,T;H^{1}_0(\O));$
\item $\rho|\u|^{2} \in L^{\infty}(0,T;L^{1}(\O));$
\item $f(t,x,\v)\geq 0, \text { for any } (t,x,\v) \in (0,T)\times \O\times \R^3;$
\item $f \in L^{\infty}(0,T;L^{\infty}(\O\times\R^3) \cap L^1(\O\times\R^3));$
\item $m_{3}f \in L^{\infty}(0,T;L^1(\O\times\R^3));$
\item For any $\varphi \in C^{1}([0,T]\times \O)$, such that $\Dv_{x}\varphi=0$, for almost everywhere $t$,
\begin{equation}
\label{2.1+}
\begin{split}
&-\int_{\O}\m_{0}\cdot\varphi(0,x)\;dx+\int_{0}^{t}\int_{\O}\bigg(-\rho\u\cdot\partial_{t}\varphi
-(\rho\u\otimes\u)\cdot\nabla\varphi
\\&\qquad\qquad\qquad\qquad\qquad\qquad\qquad
+\mu \nabla\u\cdot\nabla\varphi+\mu\rho\int_{\R^3}f(\u-\v)\cdot \varphi \,d\v\bigg)\;dxdt=0;
\end{split}
\end{equation}
\item For any $\phi \in C^1([0,T]\times\O\times\R^3)$ with compact support in $\v$, such that $\phi(T,\cdot,\cdot)=0$,
\begin{equation}
\label{2.2+}
\begin{split}
&-\int_{0}^{T}\int_{\O}\int_{\R^3}f\left({\partial_t\phi+\v\cdot\nabla_{x}\phi+\rho(\u-\v) \cdot\nabla_{\v}\phi}\right)\;dxd\v ds=\int_{\O}\int_{\R^3}f_{0}\phi(0,\cdot,\cdot)\;dxd\v;
\end{split}
\end{equation}
\item The energy inequality
\begin{equation}
\label{2.3+}
\begin{split}
&\int_{\O}\rho|\u|^{2}dx+\int_{\O}\int_{\R^3}f(1+|\v|^{2})\;d\v dx+2\int_{0}^{T}\int_{\O}\int_{\R^3} f|\u-\v|^2\; d\v dxdt\\
&\qquad +2\int_{0}^{T}\int_{\O}|\nabla\u|^2\;dxdt
\\&\leq\int_{\O}\frac{|\m_{0}|^2}{\rho_{0}}\;dx+\int_{\O}\int_{\R^3}(1+|\v|^{2})f_{0}\; d\v dx
\end{split}
\end{equation}
\text{holds for almost everywhere} $t\in[0,T].$
 \end{itemize}
 \end{Definition}
Our main result on the global weak solutions reads as follows.
 \begin{Theorem}\label{T1}
 Under the assumption \eqref{2.4}, there exists a global weak solution $(\rho,\u,f)$ to the initial-boundary value problem \eqref{1}-\eqref{boundary condition} for any $T>0$.

 \end{Theorem}

 \begin{Remark}
The same existence of global weak solutions  holds also in two-dimensional spaces.
 \end{Remark}

\bigskip

\section{Existence of Global Weak Solutions}
In this section, we are going to prove Theorem \ref{T1} in two steps.
First, we build a regularized approximation system for the original system, and solve this approximation system.
Then, we recover the original system from the approximation scheme by passing to the limit of the sequence of  solutions obtained in  the first step.

\subsection{Construction of approximation solutions}
For each $\varepsilon>0$, we define
\begin{equation*}
\theta_{\varepsilon}:=\varepsilon^{3}\theta\left(\frac{x}\varepsilon\right)
\end{equation*}
and denote $$\u_{\varepsilon}:=\u\ast\theta_{\varepsilon},$$
where  $\theta$  is the the standard mollifier satisfying
$$\theta \in C^{\infty}(\R^3), \theta\geq0, \int_{\R^3}\theta\, dx=1.$$

By \eqref{2.1++}, all values of the solution $\rho$ are bounded uniformly.
The regularity of the term $-\int_{\R^3}(\u-\v)\rho f d\v$ is not enough to solve the Navier-Stokes equation directly. Inspired by the work of \cite{H}, we introduce the following regularization function
\begin{equation*}
R_{\delta}(m_0f,m_1f)=\frac{1}{1+\delta\int_{\R^3}f\;d\v+\delta\left|\int_{\R^3}f\v d\v \right|},\quad \text{ for any fixed } \delta>0.
\end{equation*}
Clearly $$0<R_{\delta}(m_0f,m_1f)<1$$ for any $\delta >0$, and $$ R_{\delta}(m_0f,m_1f)\to1$$ as $\delta\to 0.$
For any fixed $\delta>0$, as mentioned in the introduction, the regularized force term
$$\rho R_{\delta}\int_{\R^3}(\u-\v)f\;d\v$$
consists of two terms:
$$\rho\left( R_{\delta}\int_{\R^3}f\,d\v \right)\u\;\;\;\text{ and }\;\;\rho\left(R_{\delta}\int_{\R^3}\v f\,d\v\right)$$
the first one is viewed as the work of internal force, and the second one is viewed as the external force.
The regularized external force is in $L^{2}((0,T)\times\O),$ which ensures  that the regularized Navier-Stokes equations with the work of internal force  have a smooth solution.
To keep a similar energy inequality for the approximation scheme, we need to  regularize the acceleration term  as
$$R_{\delta}(\u-\v)\rho f$$ in the Vlasov equation. Thus, we consider the following approximation problem:
\begin{equation}\label{3.1}
\rho_{t}+\Dv(\rho\u_{\varepsilon})=0,
\end{equation}
\begin{equation}\label{3.2}
\frac{\partial(\rho\u)}{\partial{t}}+\Dv(\rho\u_{\varepsilon}\otimes\u)-\mu\D\u+\nabla p +\rho\left( R_{\delta}\int_{\R^3}f\,d\v \right)\u=\rho \left(R_{\delta}\int_{\R^3}\v f\,d\v\right),
\end{equation}
\begin{equation} \label{3.3}
\Dv{\u}=0,
\end{equation}
\begin{equation}\label{3.4}
\frac{\partial f}{\partial t}+\v\cdot\nabla f+\Dv_{\v}( R_{\delta}(\u-\v)\rho f)=0.
\end{equation}
To define $\u_{\varepsilon}$ well, we need to set $$\O_{\varepsilon}=\{x\in\O, \,\,\text{dist}(x,\partial\O)>\varepsilon\}$$ for any $\varepsilon>0$ if $\O$ is smooth. Otherwise, we can choose a smooth connected domain $\O_{\varepsilon}$ such that
\begin{equation*}
\{x\in\O,\,\,\text{dist}(x,\partial \O)>\varepsilon\}\subset\O_{\varepsilon}\subset \overline{\O_{\varepsilon}}\subset\O.
\end{equation*}
We let $\hat{\u}^{\varepsilon}$ to be the truncation in $\O_{\varepsilon}$ of $\u$, and we extended it by $0$ to $\O$. We define $\u_{\varepsilon}=\hat{\u}^{\varepsilon}*\theta_{\frac{\varepsilon}{2}}$.
It is easy to find that $\u_{\varepsilon}$ is a smooth function with respect to $x$, and
$$\u_{\varepsilon}=0\;\;\text{ on }\;\partial\O\;\;\text{ and }\;\;\Dv\u_{\varepsilon}=0\;\;\text{ in }\;\R^d.$$

To impose the initial value for our approximate system, we need the following elementary variant of Hodge-de Rham decomposition (see \cite{L}):

\begin{Lemma}\label{l1+}
Let $N\geq 2$, $\rho\in L^{\infty}(\R^{N})$ such that $\rho\geq \underline{\rho}\geq 0$ almost everywhere on $\R^{N}$ for some $\underline{\rho}\in (0,\infty).$
Then there exists two bounded operators $P_{\delta}$, $Q_{\delta}$ on $L^{2}(\R^{N})$ such that for all $\m\in L^{2}(\R^{N})$, $(\m_p,\m_q)=(P_{\rho}\m,Q_{\rho}\m)$ is the unique solution in $L^{2}(\R^{N})$ of
\begin{equation*}
\m=\m_p+\m_q,\quad\quad\quad (-\D)^{-1/2}\Dv({\rho}^{-1}\m_p)=0,\quad(-\D)^{-1/2}rot(\m_q)=0.
\end{equation*}
Furthermore, if $\rho_{n}\in L^{\infty}(\R^{N})$, $\underline{\rho}\leq \rho_{n}\leq \bar{\rho}$ almost everywhere on $\R^{N}$ for some $0<\underline{\rho} \leq \bar{\rho}< \infty$ and $\rho_{n}$ converges almost everywhere to $\rho$, then $(P_{ \rho_{n} }\m_{n},Q_{\rho_{n}}\m_{n})$ converges weakly in $L^{2}(\R^{N})$ to
$(P_{\rho}\m,Q_{\rho}\m)$ whenever $\m_{n}$ converges weakly to $\m$.
\end{Lemma}
We are ready to discuss the initial conditions for the approximation scheme \eqref{3.1}-\eqref{3.4}. Before imposing initial data,
we have to point out that the initial density may be vanish in a domain: an initial vacuum may exist, and then in this case we cannot directly impose initial data on the velocity $\u$.
To remove this difficulty, we adopt the idea   from \cite{L} to define
\[
   \hat{\rho_0} = \left\{
  \begin{array}{l l}
     \rho_0, & \quad \text{if $x$ is in $\O$}\\
     1, & \quad \text{if $x$ is in $\O^c$},\\
   \end{array} \right.
 \]
define
\begin{equation*}
\begin{split}
&(\rho_0)_{\varepsilon}=\hat{\rho_0}*\theta{\varepsilon}|_{\O},
\\&(\rho_0^{\frac{1}{2}})_{\varepsilon}=\hat{\rho_0}^{\frac{1}{2}}*\theta_{\varepsilon}|_{\O},
\\&(m_0\rho_0^{-\frac{1}{2}})_{\varepsilon}=\left(m_0\rho_0^{-\frac{1}{2}}1_{\{d>2\varepsilon\}}\right)*\theta_{\varepsilon},
\end{split}
\end{equation*}
where $d=\text{dist}(x,\partial\O).$

Now we   define
\begin{equation}\label{3.5}
\rho|_{t=0}=\rho_{0}^{\varepsilon}=(\rho_{0})_{\varepsilon}+\varepsilon,
\end{equation}
which implies
\begin{equation}\label{3.6}
\varepsilon\leq \rho_{0}^{\varepsilon} \leq C_{0},
\end{equation}
where $C_{0}$ is independent on $\varepsilon,$  and
$$(\rho_{0})_{\varepsilon}=\rho_{0}\ast\theta_{\varepsilon}.$$
 Clearly, $\rho_{0}^{\varepsilon}\in C^{\infty}(\O)$, and $$\rho_{0}^{\varepsilon}\to\rho_{0}\;\text { in } L^{p}(\O)\;\;\text { for all }\; 1\leq p< \infty.$$
We define $$\rho\u|_{t=0}=\m_{0}^{\varepsilon},$$
and  $$\bar{\m}_{0}^{\varepsilon}=(\m_{0}\rho_{0}^{-1/2})_{\varepsilon}(\rho_{0}^{1/2})_{\varepsilon} \in C^{\infty}_{0}(\O).$$
It is easy to see
$$\bar{\m}_{0}^{\varepsilon}\to \m_{0} \; \text{ in } L^{2}(\O),\quad\quad\bar{\m}_{0}^{\varepsilon}(\rho_{0}^{\varepsilon})^{-1/2}\to \m_{0}\rho_{0}^{-1/2}\;\text{ in } L^{2}(\O).$$
Relying on Lemma \ref{l1+}, we decompose $\bar{\m}_{0}^{\varepsilon}$ as
\begin{equation*}\begin{split}
&\bar{\m}_{0}^{\varepsilon}=\rho_{0}^{\varepsilon}\bar{\u}_{0}^{\varepsilon}+\nabla q_{0}^{\varepsilon},\;\text{ where }\;\bar{\u}_{0}^{\varepsilon},\; q_{0}^{\varepsilon} \in C^{\infty}(\O),
\\& \Dv\bar{\u}_{0}^{\varepsilon}=0 \;\;\text{ in } \O,
\end{split}
\end{equation*}
and then
\begin{equation*}
\Dv\left(\frac{1}{\rho_{0}^{\varepsilon}}(\nabla q_{0}^{\varepsilon}-\bar{\m}_{0}^{\varepsilon})\right)=0\quad\text{ in } \O.
\end{equation*}
Letting
\begin{equation}\label{3.1+}
\begin{split}
&\m_{0}^{\varepsilon}=\rho_{0}^{\varepsilon}\u_{0}^{\varepsilon}+\nabla q_{0}^{\varepsilon}, \; \text { where }\u_{0}^{\varepsilon} \in C^{\infty}_{0}(\O),
\\& \|\u_{0}^{\varepsilon}-\bar{\u}_{0}^{\varepsilon}\|\leq \varepsilon,\quad \Dv\u_{0}^{\varepsilon}=0\quad\quad\text { in } \O.
\end{split}
\end{equation}
We have
\begin{equation}
\label{3.2+}
\m_{0}^{\varepsilon}\to \m_{0} \;\; \text{ in } L^{2}(\O),\;\;\; \m_{0}^{\varepsilon}(\rho_{0}^{\varepsilon})^{-1/2}\to \m_{0}\rho_{0}^{-1/2} \;\; \text { in } L^2(\O).
\end{equation}
Thus $$\rho\u|_{t=0}=\m_{0}^{\varepsilon}=\bar{\m}_{0}^{\varepsilon}+\rho_{0}^{\varepsilon}(\u_{0}^{\varepsilon}-\bar{\u}_{0}^{\varepsilon}),$$
and we can impose the initial condition of $\u$ as
\begin{equation}\label{3.7}
\u|_{t=0}=\u_{0}^{\varepsilon}.
\end{equation}
Finally, we impose the initial condition for $f$ as
\begin{equation}\label{3.8}
f|_{t=0}=f_{0}.
\end{equation}

We now state and prove the following existence result.
\begin{Theorem}\label{T2}
With the above notations and assumptions, there exists a solution $(\rho,\u,f)$ of \eqref{3.1}-\eqref{3.4} with the initial conditions \eqref{3.5}, \eqref{3.7} and \eqref{3.8}, and the boundary conditions \eqref{boundary condition},  such that $\rho \in C^{\infty}(\O\times[0,\infty))$, $\u \in C^{\infty}(\O\times[0,\infty))$ and $f \in L^{\infty}(\O\times\R^3\times[0,\infty)).$
\end{Theorem}

\begin{Remark}
Our approximation scheme is   inspired by Lions' work on the density-dependent Navier-Stokes equations \cite{L} and Hamdache's work on the Vlasov-Stokes equations \cite{H}.
\end{Remark}

\begin{Remark}
If the initial data $f_0$ is smooth enough, we can show that the solutions are classical solutions. In fact, we can also show the uniqueness of such solutions.
\end{Remark}

\begin{proof}[Proof of Theorem \ref{T2}]
We define $M$ as the convex set in $$C([0,T]\times \O)\times L^{2}(0,T;H^{1}_0(\O))$$ by
\begin{equation*}\begin{split}
M=\Big\{&(\bar{\rho},\bar{\u})\in C([0,T]\times \O)\times L^{2}(0,T;H^{1}_0(\O)),
\\& \varepsilon\leq\bar{\rho}\leq C_{0} \text{ in } [0,T]\times \O,\,\,\Dv{\bar{\u}}=0\;\; \text{   almost everywhere on } (0,T)\times\O,\\& \|\bar{\u}\|_{L^{2}(0,T;H^{1}_0(\O))}\leq K \Big\},
\end{split}\end{equation*}
where $K>0$ is to be determined.
Here we define a map $T$ from $M$ into itself as
\begin{equation*}
T(\bar{\rho},\bar{\u})=(\rho,\u).
\end{equation*}

As a first step, we consider the following initial-value problem:
\begin{equation}\label{3.9}
\frac{\partial{\rho}}{\partial{t}}+\Dv(\bar{\u}_{\varepsilon}\rho)=0, \quad\quad\rho|_{t=0}=\rho_{0}^{\varepsilon}, \quad\quad
\end{equation}
in $(0,T)\times\O$,
where $\bar{\u}_{\varepsilon}=\bar{\u}\ast\theta_{\varepsilon}.$ The construction of $\bar{\u}_{\varepsilon}$ implies that
$\bar{\u}_{\varepsilon}\in L^{2}(0,T;C^{\infty}(\O)),$
and
$\Dv{\bar{\u}_{\varepsilon}}=0$ in $(0,T)\times \O.$
The solution of \eqref{3.9} can be written in terms of characteristics:
\begin{equation}\label{3.10}
\frac{dX}{ds}=\bar{\u}_{\varepsilon}(X,s),\quad\; X(x;x,t)=x,\quad x\in \O,\quad t\in [0,T].
\end{equation}
By the properties of $\bar{\u}_{\varepsilon}\in L^{2}(0,T;C^{\infty}(\O))$, and the basic theory of ordinary differential equations,
we know that there exists a unique solution $X$ of \eqref{3.10}.
Therefore, we have
\begin{equation*}
\rho(t,x)=\rho_{0}^{\varepsilon}(X(0;t,x)), \;\text { for all }t\in [0,T],\quad x\in \O.
\end{equation*}
It is clear that $\varepsilon\leq \rho\leq C_{0}$ in $[0,T]\times \O.$
Since $\bar{\u}_{\varepsilon}\in L^{2}(0,T;C^{\infty}(\O))$, then $\rho(t,x)$ lies in $C([0,T];C^{\infty}(\O)).$
By \eqref{3.9} and the properties of $\bar{\u}_{\varepsilon}$, we have $\frac{\partial{\rho}}{\partial{t}}\in L^{2}(0,T;C^{\infty}(\O)).$ Thus, $\rho\text{ and }\frac{\partial{\rho}}{\partial{t}}$ are bounded in these spaces uniformly in $(\bar{\rho},\bar{\u})\in M$. In particular, by the Aubin-Lions lemma, the set of $\rho$ built in this way is clearly compact in $C([0,T]\times \O).$

The second step is to build $\u$ by solving the following problem:
\begin{equation}\label{3.14}
\begin{split}
&\rho\frac{\partial{\u}}{\partial{t}}+\rho\bar{\u}_{\varepsilon}\nabla\u-\D\u+\nabla p+\rho\left( R_{\delta}\int_{\R^3}f\,d\v \right)\u=\rho \left(R_{\delta}\int_{\R^3}\v f\,d\v\right),
\\& \Dv\u=0,\quad\;\; \u|_{t=0}=\u_{0}^{\varepsilon},\quad\quad\Dv\u_0^{\varepsilon}=0,
\end{split}
\end{equation}
in$ (0,T)\times\O$.
Let $$e=R_{\delta}\int_{\R^3}f\;d\v\geq 0,\quad g=R_{\delta}\int_{\R^3}\v f\;d\v. $$
Multiplying $\u$ on both sides of \eqref{3.14},
one obtains   the following energy  equality related to \eqref{3.14}:
\begin{equation*}
\partial_t\int_{\O}\frac{1}{2}\rho|\u|^2\;dx+\int_{\O}|\nabla\u|^2\;dx+\int_{\O}e\rho|\u|^2\;dx=\int_{\O}\rho g\u \;dx.
\end{equation*}
The right-hand side of above energy  equality is bounded by
\begin{equation*}
\begin{split}
&\int_{\O}\rho g \u\;dx\leq (\int_{\O}\rho|g|^2\;dx)^{\frac{1}{2}}\cdot(\int_{\O}\rho|\u|^2\;dx)^{\frac{1}{2}}
\\& \leq\|\rho_0\|_{L^{\infty}(\O)}^{\frac{1}{2}}\|g\|_{L^2(\O)}\|\sqrt{\rho}\u\|_{L^2(\O)}.
\end{split}
\end{equation*}
In conclusion, we obtain for all $t\in(0,T)$,
\begin{equation*}
\begin{split}
&\frac{1}{2}\int_{\O}\rho|\u|^2dx+\int_0^{t}\int_{\O}|\nabla\u|^2\;dxdt+\int_0^{t}\int_{\O}e\rho|\u|^2\;dxdt
\\&\leq C\int_0^t\int_{\O}|g|^2\,dxdt+C\int_0^t\int_{\O}\rho|\u|^2\,dxdt+\frac{1}{2}\int_{\O}\frac{|\m_0|^2}{\rho_0}\,dx.
\end{split}
\end{equation*}
Applying Gronwall inequality, we obtain
\begin{equation}
\begin{split}
\label{3.2++}
&\sup_{t\in(0,T)}\|\rho|\u|^2\|_{L^{1}(\O)}\leq C,
\\&\|\sqrt{e\rho}\u\|_{L^{2}(0,T;\O)}\leq C,
\\& \|\u\|_{L^{2}(0,T;H^1_0(\O))}\leq C,
\end{split}
\end{equation}
where $C$ denotes various constant which depend only on $T,\O,\varepsilon,\delta$ and bounds on $\|\rho_0\|_{L^{\infty}(\O)},$
$\|\rho_0|\u|^{2}\|_{L^{1}(\O)}.$

Rewriting \eqref{3.14} as follows
\begin{equation}
\begin{split}
\label{3.15}
& c\frac{\partial{\u}}{\partial t}+b\cdot\nabla\u-\D\u+\nabla p+a\u=h,
\\&\;\; \Dv\u=0, \quad\; \u|_{t=0}=\u_{0}^{\varepsilon},\quad\Dv\u_{0}^{\varepsilon}=0, \quad
\end{split}
\end{equation}
in $(0,T)\times \O,$ where
\begin{equation*}
\begin{split}
& c\in L^{\infty}((0,T)\times\O),\quad b\in L^{2}(0,T;L^{\infty}(\O)),\quad a\in L^{\infty}((0,T)\times\O),
\\ &h\in L^{\infty}((0,T)\times\O),\quad c\geq\delta >0.
\end{split}
\end{equation*}
To continue our proof, we need the following lemma:
\begin{Lemma}\label{l1}
There exists a unique solution $\u$ of \eqref{3.15} with the following regularity:
\begin{equation}
\label{3.16}
\u\in L^{2}(0,T;H^{2}(\O))\cap C([0,T];H^{1}_0(\O));\quad \nabla p, \frac{\partial {\u}}{\partial t} \in L^{2}((0,T)\times\O).
\end{equation}
\end{Lemma}
\begin{proof}
First, we multiply \eqref{3.15} by $\frac{\partial {\u}}{\partial t}$ and use integration by parts over $\O$ to obtain:
\begin{equation*}
\begin{split}
&\delta \int_{\O}\left|\frac{\partial {\u}}{\partial t}\right|^{2}dx+\frac{d}{dt}\int_{\O}\frac{1}{2}|\nabla\u|^2dx
\\ & \leq\int_{\O}\left(|h||\u_t|+|b||\nabla\u||\u_t|+|a||\u||\u_t|\right)dx
\end{split}
\end{equation*}
Using the Cauchy-Schwarz inequality and embedding inequality, one deduces that
\begin{equation}\label{3.17}
\begin{split}
&\frac{\delta}{2} \int_{\O}\left|\frac{\partial {\u}}{\partial t}\right|^{2}dx+\frac{d}{dt}\int_{\O}\frac{1}{2}|\nabla\u|^2dx
\\ & \leq C \left(1+\|b\|_{L^{\infty}(\O)}^2+\|a\|_{L^{\infty}(\O)}^2\lambda_0\right)\int_{\O}|\nabla\u|^2dx,
\end{split}
\end{equation}
where $\lambda_0$ is a constant from the Sobloev inequality.
By the regularity of $a,b$ and Gronwall's inequality, we deduce that
\begin{equation*}
\u \in L^{\infty}(0,T;H^{1}_{0}(\O)),\;\; \frac{\partial{\u}}{\partial t} \in L^{2}((0,T)\times\O).
\end{equation*}
We rewrite \eqref{3.15} as follows
\begin{equation}
\label{3.18}
\begin{split}
&-\D\u+\nabla p=h-c\u_t-b\cdot\nabla\u-a\u,
\\& \Dv\u=0,
\end{split}
\end{equation}
in $\O\times (0,T),$ and
$\u \in H^{1}_{0}(\O).$
Let $\widetilde{h}=h-c\u_t-b\cdot\nabla\u-a\u,$ and $$\widetilde{h}\in L^{2}(0,T;\O),$$
thus we have
\begin{equation*}
\begin{split}
&-\D\u+\nabla p=\widetilde{h},
\\&\Dv\u=0.
\end{split}
\end{equation*}
By the regularity of $\u$ and $\widetilde{h},$
we conclude that $p$ is bounded in $L^{2}((0,T);H^{-1}(\O)).$
We deduce that $p$ is bounded in $L^{2}((0,T)\times\O)$ if we normalize $p$ by imposing $$\int_{\O}p\;dx=0,\;\;\text{ almost everywhere } t\in (0,T).$$
To normalize $p$, we refer the readers to \cite{L,T} for more details.
Also we conclude that $\u$ is bounded in $L^{2}(0,T;H^{2}(\O))$ by the classical regularity on Stokes equation.
Thus, we proved the regularity of \eqref{3.16}.
The existence and uniqueness of \eqref{3.15} follows from the Lax-Milgram theorem, see for example \cite{CF}.
\end{proof}
By Lemma \ref{l1}, there exists a unique solution to \eqref{3.14} with the regularity of \eqref{3.16}. By the Aubin-Lions Lemma, $\u$ is compact in $L^{2}(0,T;H^{1}_0(\O)).$ This, with the help of compactness of $\rho$ in $M$, implies that the mapping $T$ is compact in $M$.

To find the fixed point of map $T$ by the Schauder theorem, it remains to find $K>0$ such that
\begin{equation*}
\|\u\|_{L^{2}(0,T;H^1_0(\O))}\leq K.
\end{equation*}
By \eqref{3.2++},
we have
\begin{equation*}
\|\u\|_{L^{2}(0,T;H^1_0(\O))}\leq K',
\end{equation*}
this $K'$ only depends on initial data. Thus, we can choose $K=K'+1.$

Following the same argument of the proof of Lemma \ref{l1}, we deduce that
$$\u \in L^{p}(0,T;W^{2,p}(\O)),\quad \frac{\partial{\u}}{\partial t} \in L^{p}((0,T)\times\O)$$ for all $1<p<\infty.$ With such regularity of $\u$, we can bootstrap and obtain more time regularity on $\u_{\varepsilon}$  and
then on $\rho$ and thus more regularity on $\u.$

In the third step, we would like to find the solutions to the following nonlinear Vlasov equation:
\begin{equation}
\begin{split}
\label{3.20}
&\frac{\partial f}{\partial t}+\v\nabla_{x} f+\Dv_{\v}(R_{\delta}(\u-\v)\rho f)=0,
\\ & f(0,x,\v)=f_{0}(x,\v),\;\;f(t,x,\v)=f(t,x,\v^*), \;\text{ for } x\in \partial\O,\; \v\cdot\nu(x)<0.
\end{split}
\end{equation}
where $\u, \rho$ are smooth functions obtained in step 2. The existence and uniqueness for the above Vlasov equation can be obtained as in \cite{BP,DL}.

Thus we have proved Theorem \ref{T2}.
\end{proof}

\begin{Remark}\label{r2}
The solutions $(\rho,\u,f)$ obtained in Theorem \ref{T2} satisfy the following energy inequality
\begin{equation}
\begin{split}
\label{3.21}
&\frac{d}{dt}\left(\int_{\O}\frac{1}{2}\rho|\u|^2\;dx+\int_{\O}\int_{\R^3}\frac{1}{2}f(1+|\v|^2)\;dxd\v\right)
\\&+\int_{\O}|\nabla\u|^2\;dx+\int_{\O}\int_{\R^3}R_{\delta}\rho f(\u-\v)^{2} \;dxd\v \leq 0.
\end{split}
\end{equation}
The energy inequality will be crucial in deriving a prior estimates on the solutions $(\rho,\u,f)$ of the approximate  system of equations.
\end{Remark}

\subsection{Pass to the limit as $\varepsilon\to 0$.}
 The objective of this section is to recover the original system from the approximation scheme \eqref{3.1}-\eqref{3.4} upon letting $\varepsilon$ goes to $0$. Here and below, we denote by $(\rho^{\varepsilon},\u^{\varepsilon},f^{\varepsilon})$ the solution constructed in Theorem \ref{T2}.

We take $\beta \in C(\O,\R^3)$, use \eqref{3.1} and \eqref{3.3} to find that
$\int_{\O}\beta(\rho^{\varepsilon})dx$ is independent of time $t$, that is,
\begin{equation}
\label{3.6+}
\int_{\O}\beta(\rho^{\varepsilon})\;dx=\int_{\O}\beta(\rho_{0}^{\varepsilon})\;dx \quad \text{ for all } t\in(0,\infty).
\end{equation}
Observing that $(\rho^{\varepsilon},\u^{\varepsilon},f^{\varepsilon})$ satisfies \eqref{3.21}, one obtains
\begin{equation*}
\begin{split}
&\frac{d}{dt}\left(\int_{\O}\frac{1}{2}\rho^{\varepsilon}|\u^{\varepsilon}|^2\;dx+\int_{\O}\int_{\R^3}\frac{1}{2}f^{\varepsilon}(1+|\v|^2)\;dxd\v\right)
\\&+\int_{\O}|\nabla\u^{\varepsilon}|^2\;dx+\int_{\O}\int_{\R^3}R_{\delta}\rho^{\varepsilon} f^{\varepsilon}(\u^{\varepsilon}-\v)^{2} \;dxd\v \leq 0.
\end{split}
\end{equation*}
Integrating it from $0$ to $t$, we have
\begin{equation}
\begin{split}
\label{3.22}
&\int_{\O}\frac{1}{2}\rho^{\varepsilon}|\u^{\varepsilon}|^2\;dx+\int_{\O}\int_{\R^3}\frac{1}{2}f^{\varepsilon}(1+|\v|^2)\;dxd\v
\\&+\int_{0}^{t}\int_{\O}|\nabla\u^{\varepsilon}|^2\;dxdt+\int_{0}^{t}\int_{\O}\int_{\R^3}R_{\delta}\rho^{\varepsilon} f^{\varepsilon}|\u^{\varepsilon}-\v|^{2} \;dxd\v dt
\\& \leq\frac{1}{2}\int_{\O}\rho_{0}^{\varepsilon}|\u_{0}^{\varepsilon}|^2\;dx+\frac{1}{2}\int_{\O}\int_{\R^3}(1+|\v|^2)f_{0} \;dx d\v
\end{split}
\end{equation}
for all $t>0.$
By \eqref{3.22}, one obtains the following estimates:
\begin{equation}
\begin{split}
\label{3.23}
&\|\u^{\varepsilon}\|_{L^{2}(0,T;H^{1}_{0}(\O))}\leq C,
\\& \sup_{0\leq t \leq T}\|\rho^{\varepsilon}|\u^{\varepsilon}|^2\|_{L^{1}(\O)}\leq C,
\end{split}
\end{equation}
where $C$ denotes a generic positive constant independent of
$\varepsilon$.

By \eqref{3.5} and \eqref{3.6}, we assume that, up to the extraction of subsequences,
\begin{equation}
\label{3.7+}
\rho^{\varepsilon}\to \rho \;\;\text{ in } C([0,T];L^{p}(\O)) \text{ for any }1\leq p<\infty.
\end{equation}

We denote by $\u$   the weak limit of $\u^{\varepsilon}$ in $L^{2}(0,T;H^{1}_0(\O))$ due to \eqref{3.23}. By the compactness of the embedding $L^{p}(\O)\hookrightarrow W^{-1,2}_0(\O)$ for any $p>6/5,$ one deduces from \eqref{3.7+}:
\begin{equation}
\label{3.8+}
\rho^{\varepsilon} \to \rho \;\;\text{ in }  C([0,T];W^{-1,2}_0(\O)).
\end{equation}
This, together with \eqref{3.23}, yields
\begin{equation*}
\rho^{\varepsilon}\u^{\varepsilon}\to \rho\u \;\; \text{ in } \mathcal{D}'((0,T)\times\O).
\end{equation*}
Let a function $g\in C([0,T];L^{p}(\O))$ for any $1<p<\infty$ satisfy $g(0)=0$ on $\O$ and
\begin{equation*}
\frac{\partial g}{\partial t}+\Dv(g\u)=0 \;\;\text{ in } \mathcal{D}'((0,T)\times\O),
\end{equation*}
then $g\equiv0$, which implies the uniqueness of the density $\rho$ when $\u$ is fixed.
Thus we have proved that $\rho$ is the solution to \eqref{1}.

We now estimate $m_{0}f^{\varepsilon}$:
\begin{equation*}
\begin{split}
m_{0}f^{\varepsilon}&=\int_{\R^3}f^\varepsilon\;d\v=\int_{|\v|<r}f^{\varepsilon}\;d\v+\int_{|\v|\geq r}f^{\varepsilon}\;d\v
\\& \leq C\|f^\varepsilon\|_{L^{\infty}}r^3+\frac{1}{r^k}\int_{|\v|\geq r}|\v|^k f^{\varepsilon} \;d\v
\end{split}
\end{equation*}
for all $k\geq 0.$
Taking $r=(\int_{\R^3}|\v|^k f^\varepsilon d\v)^{\frac{1}{k+3}}$, we have
\begin{equation*}
m_{0}f^{\varepsilon}\leq C\|f^\varepsilon\|_{L^{\infty}}\left(\int_{\R^3}|\v|^k f^\varepsilon d\v\right)^{\frac{3}{k+3}}+\left(\int_{\R^3}|\v|^k f^\varepsilon d\v\right)^{\frac{3}{k+3}}.
\end{equation*}
Letting $k=3$, then
\begin{equation*}
\|m_{0}f^{\varepsilon}\|_{L^{2}(\O)}\leq C(\|f^\varepsilon\|_{L^{\infty}}+1)\left(\int_{\O}\int_{\R^3}|\v|^3 f^{\varepsilon}d\v\right)^{1/2}.
\end{equation*}
Thanks to Lemma \ref{l2+}, we conclude that $m_{3}f^{\varepsilon}$ is bounded in $L^{\infty}(0,T;L^{1}(\O))$.
This yields
\begin{equation}
\label{3.9+}
\|m_{0}f^{\varepsilon}\|_{L^{\infty}(0,T;L^{2}(\O))}\leq C.
\end{equation}
Following the same argument, one deduces that
\begin{equation}
\label{3.10+}
\|m_{1}f^{\varepsilon}\|_{L^{\infty}(0,T;L^{\frac{3}{2}}(\O))}\leq C.
\end{equation}
Using the fact $R_{\delta} \leq 1$, we see that
\begin{equation}
\label{3.11+}
\begin{split}
&\| \rho R_{\delta}m_{0}f^{\varepsilon}\u^{\varepsilon}\|_{L^{2}(0,T;L^{3/2}(\O))}
\\&\leq C\|\rho_0\|_{L^{\infty}((0,T)\times\O)}\|m_{0}f^{\varepsilon}\|_{L^{\infty}(0,T;L^{2}(\O))}\cdot\|\u^{\varepsilon}\|_{L^{2}(0,T;L^{6}(\O))},
\end{split}
\end{equation}
and
\begin{equation}
\label{3.12+}
\|\rho R_{\delta}m_{1}f^{\varepsilon}\|_{L^{\infty}(0,T;L^{3/2}(\O))}\leq C \|\rho_0\|_{L^{\infty}((0,T)\times\O)}\|m_{1}f^{\varepsilon}\|_{L^{\infty}(0,T;L^{\frac{3}{2}}(\O))}.
\end{equation}
Observing $$\rho^{\varepsilon} R_{\delta}\int_{\R^3}(\u^{\varepsilon}-\v) f^{\varepsilon}d\v=\rho^{\varepsilon} R_{\delta}m_{0}f^{\varepsilon}\u^{\varepsilon}-\rho R_{\delta}m_{1}f^{\varepsilon},$$
and using \eqref{3.11+} and \eqref{3.12+}, we obtain that
\begin{equation*}
\rho^{\varepsilon} R_{\delta}\int_{\R^3}(\u^{\varepsilon}-\v)f^{\varepsilon}\;d\v \quad\text{ is bounded in } L^{2}(0,T;L^{3/2}(\O)).
\end{equation*}
Since
\begin{equation*}
\frac{\partial(\rho^{\varepsilon}\u^{\varepsilon})}{\partial{t}}=-\Dv(\rho^{\varepsilon}\u_{\varepsilon}\otimes\u^{\varepsilon})
+\D\u^{\varepsilon}+\nabla p+\rho R_{\delta}\int_{\R^3}(\u^{\varepsilon}-\v)f^{\varepsilon}d\v,
\end{equation*}
and in particular, $\nabla\u^{\varepsilon}$ is  bounded in $L^{2}((0,T)\times\O)$ and
\begin{equation*}
\rho^{\varepsilon} R_{\delta}\int_{\R^3}(\u^{\varepsilon}-\v)f^{\varepsilon}d\v \;\;\text{ is bounded in } L^{2}(0,T;L^{3/2}(\O))
\end{equation*}
while $\rho^{\varepsilon}\u_{\varepsilon}\otimes\u^{\varepsilon}$ is bounded in $L^{2}(0,T;L^{\frac{3}{2}}(\O))$, one obtains that
\begin{equation*}
\frac{\partial(\rho^{\varepsilon}\u^{\varepsilon})}{\partial{t}} \;\; \text{ is bounded in } L^{2}(0,T;H^{-1}(\O)).
\end{equation*}
By Theorem 2.4 of \cite{L}, we obtain
\begin{equation*}\sqrt{\rho^{\varepsilon}}\u^{\varepsilon} \to \sqrt{\rho}\u \;\; \text{ in }L^{p}(0,T;L^{r}(\O))
 \end{equation*}
 for $2<p<\infty$ and $1\leq r<\frac{6p}{3p-4}$, and thus
 \begin{equation*}
 \rho^{\varepsilon}\u^{\varepsilon}\to\rho\u\quad \text{ in } L^{p}(0,T;L^{r}(\O))
 \end{equation*}
 for the same values of $p,r.$

Applying Lemma \ref{l3+} to \eqref{3.4}, we obtain
\begin{equation}
\label{3.13+}
m_{0}f^{\varepsilon}\to m_{0}f,\quad m_{1}f^{\varepsilon}\to m_{1}f\quad\text{ for almost everywhere }(t,x).
\end{equation}
By \eqref{3.9+} and \eqref{3.10+}, the relation \eqref{3.13+} can be strengthened to the following statement:
\begin{equation}
\label{3.14++}
\begin{split}
&m_{0}f^{\varepsilon}\to m_{0}f  \quad\text{ strongly in } L^{\infty}(0,T;L^{2}(\O)),
\\& m_{1}f^{\varepsilon}\to m_{1}f \quad\text{ strongly in } L^{\infty}(0,T;L^{3/2}(\O)).
\end{split}
\end{equation}

By \eqref{3.7+},
we have
\begin{equation}
\label{3.15+}
\rho^{\varepsilon}m_{0}f^{\varepsilon}\to \rho m_{0}f\;\;\text{ strongly in } L^{\infty}(0,T;L^{\frac{2p}{2+p}}(\O)),
\end{equation}
and
\begin{equation}
\label{3.16+}
\rho^{\varepsilon}m_{1}f^{\varepsilon}\to\rho m_{1}f \;\;\text{ strongly in } L^{\infty}(0,T;L^{\frac{3p}{2p+3}}(\O)).
\end{equation}
Thanks to \eqref{3.15+}-\eqref{3.16+} and the weak convergence of
$\u^{\varepsilon}$ in $L^{2}(0,T;H^{1}_0(\O))$, one has
\begin{equation}
\label{forceterm}
R_{\delta}\int_{\R^3}(\u^{\varepsilon}-\v)\rho^{\varepsilon}f^{\varepsilon}
\;d\v\to R_{\delta} \int_{\R^3}(\u-\v)\rho f\;d\v\;\;\text{ in the sense of
distributions}.
\end{equation}

The next step is to deal with the convergence of
$\Dv_{\v}(R_{\delta}\rho^{\varepsilon}(\u^{\varepsilon}-\v)f^{\varepsilon}).$
Let $\phi(\v)\in\mathcal {D}(\R^3)$   be a test function, we
want to show
\begin{equation}\label{3.31}
\begin{split}
&\lim_{\varepsilon\to 0}\left(\int_{\O}\int_{\R^3}(R_{\delta}\rho^{\varepsilon}(\u^{\varepsilon}-\v)f^{\varepsilon})\nabla_{\v}\phi\;d\v dx\right)=\\&\int_{\O}\int_{\R^3}(R_{\delta}\rho(\u-\v)f)\nabla_{\v}\phi\;d\v dx,
\end{split}
\end{equation}
which can be reached by \eqref{forceterm}.

We consider a test function $\varphi\in C^{1}_{0}([0,T]\times \O)$ such that $\Dv\varphi=0,$ and a test function $\phi \in C^1([0,T]\times\O\times\R^3)$ with compact support in $\v$, such that $\phi(T,\cdot,\cdot)=0$. The weak formulation associated with the approximation scheme \eqref{3.1}-\eqref{3.4} should be

\begin{equation}
\label{3.34}
\begin{split}
&-\int_{\O}\rho_{0}^{\varepsilon}\u_{0}^{\varepsilon}\cdot\varphi(0,x)\;dx+\int_{0}^{t}\int_{\O}\{-\rho^{\varepsilon}\u^{\varepsilon}\cdot\partial_{t}\varphi
-(\rho\u^{\varepsilon}\otimes\u^{\varepsilon})\cdot\nabla\varphi
\\&+\nabla\u^{\varepsilon}\cdot\nabla\varphi+\varphi\cdot R_{\delta}\int_{\R^3}(\u^{\varepsilon}-\v)\rho^{\varepsilon}f^{\varepsilon} d\v \}\;dxdt=0;
\end{split}
\end{equation}

and

\begin{equation}
\label{3.35}
\begin{split}
&-\int_{0}^{T}\int_{\O}\int_{\R^3}f^{\varepsilon}\left(\partial_t\phi+\v\cdot\nabla_{x}\phi+R_{\delta}(\u^{\varepsilon}-\v)\rho^{\varepsilon} \cdot\nabla_{\v}\phi\right)\;dxd\v ds
\\&=\int_{\O}\int_{\R^3}f_{0}\phi(0,\cdot,\cdot)\;dxd\v.
\end{split}
\end{equation}
By \eqref{3.1+}-\eqref{3.2+}, we have
\begin{equation*}
\int_{\O}\rho_{0}^{\varepsilon}\u_{0}^{\varepsilon}\cdot\varphi\; dx=\int_{\O}\m_{0}^{\varepsilon}\cdot\varphi dx\to \int_{\O}\m_{0}\cdot\varphi\; dx \quad\text{ as } \varepsilon\to 0,
\end{equation*}
for all test functions $\varphi.$

All the above convergence results in this subsection allow us to recover \eqref{2.1+}-\eqref{2.2+} by passing to the limits in \eqref{3.34} and \eqref{3.35} as $\varepsilon\to 0.$

From \eqref{3.22},   the solution $(\rho^{\varepsilon},\u^{\varepsilon},f^{\varepsilon})$ satisfies  the following:
\begin{equation*}
\begin{split}
&\int_{\O}\frac{1}{2}\rho^{\varepsilon}|\u^{\varepsilon}|^2\;dx+\int_{\O}\int_{\R^3}\frac{1}{2}f^{\varepsilon}(1+|\v|^2)\;dxd\v
\\&+\int_{0}^{t}\int_{\O}|\nabla\u^{\varepsilon}|^2\;dxdt+\int_{0}^{t}\int_{\O}\int_{\R^3}R_{\delta}\rho^{\varepsilon} f^{\varepsilon}|\u^{\varepsilon}-\v|^{2} \;dxd\v dt
\\& \leq\frac{1}{2}\int_{\O}\rho_{0}^{\varepsilon}|\u_{0}^{\varepsilon}|^2\;dx+\frac{1}{2}\int_{\O}\int_{\R^3}(1+|\v|^2)f_{0} \;dx d\v.
\end{split}
\end{equation*}
The difficulty of passing the limit for the energy inequality is the convergence of the term $\int_0^t\int_{\O\times\R^3}R_{\delta}\rho^{\varepsilon}f^{\varepsilon}|\u^{\varepsilon}-\v|^2\,d\v dxdt$.
We follow the same way as in \cite{H,YU1} to treat the term as follows
\begin{equation}
\label{energyinequality}
\begin{split}
&\int_0^T\int_{\O\times\R^3}R_{\delta}\rho^{\varepsilon}f^{\varepsilon}|\u^{\varepsilon}-\v|^2\,d\v\, dx\,dt
\\&\quad\quad\quad\quad\quad\quad=\int_0^T\int_{\O\times\R^3}\left(R_{\delta}\rho^{\varepsilon}f^{\varepsilon}|\u^{\varepsilon}|^2-2R_{\delta}\rho^{\varepsilon}f^{\varepsilon}\u^{\varepsilon}\v+R_{\delta}\rho^{\varepsilon}f^{\varepsilon}|\v|^2\right)\,dx d\v dt.
\end{split}
\end{equation}
By the embedding inequality, we have
\begin{equation}
\label{convergenceofcutu}
\u^{\varepsilon}\to\u\;\;\text{ weakly in } L^{2}(0,T;L^6(\O)).
\end{equation}
 By \eqref{3.15+}, \eqref{convergenceofcutu}, we deduce that
\begin{equation*}
R_{\delta}\rho^{\varepsilon}m_0f^{\varepsilon}|\u^{\varepsilon}|^2\to R_{\delta}\rho m_0f|\u|^2
\;\;\;\text{ weakly in } L^1(0,T;\O)
\end{equation*}
as $\varepsilon\to\infty.$
Similarly,
\begin{equation*}
 R_{\delta}\rho^{\varepsilon}m_1f^{\varepsilon}\u^{\varepsilon}\to
 R_{\delta}\rho m_1f\u\;\;\;\text{ weakly in } L^1(0,T;\O)
\end{equation*}
as $\varepsilon\to 0$

Finally, let us look at the terms:
\begin{equation*}
\begin{split}
&\left|\int_0^t\int_{\O}\int_{\R^3}R_{\delta}\rho^{\varepsilon}f^{\varepsilon}|\v|^2 \,d\v dxdt-\int_0^t\int_{\O}\int_{\R^3}R_{\delta}\rho f|\v|^2 \,d\v dxdt\right|\\
&\leq C\left|\rho^{\varepsilon}-\rho|\right|_{L^{\infty}}\int_0^T\int_{\O}m_2f^{\varepsilon}dxdt+C\|\rho\|_{L^{\infty}}\int_0^t\int_{\O}(m_2f-m_2f^{\varepsilon})dxdt
\\&=I_1+I_2.
\end{split}
\end{equation*}
It is clear to see that  $I_1\to 0$ as $\varepsilon\to0.$
For the term $I_2$,
because
 $$f^{\varepsilon}\rightharpoonup f\quad\quad\text{ weak star in } L^{\infty}(0,T;L^{p}(\O\times\R^3))$$
 for all $p\in(1,\infty]$ and $m_3f^\varepsilon$ is bounded in $L^{\infty}(0,T;L^1(\O))$, then for any fixed $r>0,$ we have
  \begin{equation*}
  \int_0^T\int_{\O\times\R^3}f^{\varepsilon}|\v|^2\,dx d\v dt=\int_0^T\int_{\O\times\R^3}\chi(|\v|<r)|\v|^2f^{\varepsilon}\,dx d\v dt+O(\frac{1}{r})
  \end{equation*}
  uniformly in $\varepsilon$ where $\chi$ is the characteristic function of the ball of $\R^3$ of radius $r$. Letting $\varepsilon\to 0,$ then $r\to\infty$, we find \begin{equation*}
  \int_0^T\int_{\O\times\R^3}f^\varepsilon|\v|^2\,dx d\v dt\to   \int_0^T\int_{\O\times\R^3}f|\v|^2\,dx d\v dt,
  \end{equation*}
  which means $I_2\to 0$ as $\varepsilon\to0.$
  Thus, we have proved
  \begin{equation}
  \label{energyinequalityconvergence2}
  \int_0^T\int_{\O\times\R^3}R_{\delta}\rho^{\varepsilon}f^{\varepsilon}|\u^{\varepsilon}-\v|^2\,d\v dxdt\to \int_0^T\int_{\O\times\R^3}R_{\delta}\rho f|\u-\v|^2\,d\v dx dt
  \end{equation}
  as $\varepsilon\to \infty.$

We observe that
\begin{equation}
\label{3.38}
\begin{split}
&\int_{\O}\rho_{0}^{\varepsilon}|\u_{0}^{\varepsilon}|^2dx=\int_{\O}\frac{1}{\rho_{0}^{\varepsilon}}|\m_{0}^{\varepsilon}-\nabla q_{0}^{\varepsilon}|^2\;dx
\\&=\int_{\O}\left(\frac{|\m_{0}^{\varepsilon}|^{2}}{\rho_{0}^{\varepsilon}}+\frac{|\nabla q_{0}^{\varepsilon}|^2}{\rho_{0}^{\varepsilon}}-\frac{2}{\rho_{0}^{\varepsilon}}(\rho_{0}^{\varepsilon}\u_{0}^{\varepsilon}+\nabla q_{0}^{\varepsilon})\cdot\nabla q_{0}^{\varepsilon}\right)\; dx
\\& =\int_{\O}\left(\frac{|\m_{0}^{\varepsilon}|^{2}}{\rho_{0}^{\varepsilon}}-2 \u_{0}^{\varepsilon}\cdot\nabla q_{0}^{\varepsilon}-\frac{|\nabla q_{0}^{\varepsilon}|^2}{\rho_{0}^{\varepsilon}}\right)\;dx,
\end{split}
\end{equation}
where we used Lemma \ref{l1+}.

Using $\Dv \u_{0}^{\varepsilon}=0,$ one obtains
\begin{equation}
\label{3.41}
\int_{\O}\rho_{0}^{\varepsilon}|\u_{0}^{\varepsilon}|^2\;dx+\int_{\O}\frac{|\nabla q_{0}^{\varepsilon}|^2}{\rho_{0}^{\varepsilon}}dx=\int_{\O}\frac{|\m_{0}^{\varepsilon}|^2}{\rho_{0}^{\varepsilon}}\;dx.
\end{equation}

Letting $\varepsilon\to0$, using \eqref{3.2+}, \eqref{3.22}, \eqref{energyinequalityconvergence2}, \eqref{3.41},  and the weak convergence of $(\rho^{\varepsilon},\u^{\varepsilon},f^{\varepsilon})$, we obtain
\begin{equation*}
\begin{split}
&\int_{\O}\frac{1}{2}\rho|\u|^2\;dx+\int_{\O}\int_{\R^3}\frac{1}{2}f(1+|\v|^2)\;dxd\v
\\&+\int_{0}^{t}\int_{\O}|\nabla\u|^2\;dxdt+\int_{0}^{t}\int_{\O}\int_{\R^3}R_{\delta}\rho f|\u-\v|^{2} \;dxd\v dt
\\& \leq\frac{1}{2}\int_{\O}\frac{|m_{0}|^2}{\rho_0}\;dx+\frac{1}{2}\int_{\O}\int_{\R^3}(1+|\v|^2)f_{0} \;dx d\v.
\end{split}
\end{equation*}

So far, we have proved the following result:
\begin{Proposition}
\label{P1}
For any $T>0$, there is a weak solution $(\rho^{\delta},\u^{\delta},f^{\delta})$ to the following system:
\begin{equation*}
\begin{split}
&\rho_{t}+\Dv(\rho\u)=0,
\\& \frac{\partial(\rho\u)}{\partial{t}}+\Dv(\rho\u\otimes\u)-\mu\D\u+\nabla p=\rho R_{\delta}\int_{\R^3}(\u-\v)f\;d\v,
\\& \Dv{\u}=0,
\\& \frac{\partial f}{\partial t}+\v\cdot\nabla f+\Dv_{\v}( R_{\delta}(\u-\v)\rho f)=0.
\end{split}
\end{equation*}
with the initial data $\u(0,x)=\u_0$ and $f(0,x,\v)=f_0(x,\v)$, and boundary conditions
\begin{equation*}
\begin{split}
&\u(t,x)=0\;\;\; \text{on}\;\;\partial \O,
\\&f(t,x,\v)=f(t,x,\v^{*})\;\;\text{ for }\; x\in \partial \O,\v\cdot\nu(x)<0.
\end{split}
\end{equation*}
In additional, the solution satisfies the following energy inequality:
\begin{equation*}
\begin{split}
&\int_{\O}\frac{1}{2}\rho|\u|^2\;dx+\int_{\O}\int_{\R^3}\frac{1}{2}f(1+|\v|^2)\;dxd\v
\\&+\int_{0}^{t}\int_{\O}|\nabla\u|^2\;dxdt+\int_{0}^{t}\int_{\O}\int_{\R^3}R_{\delta}\rho f|\u-\v|^{2} \;dxd\v dt
\\& \leq\frac{1}{2}\int_{\O}\frac{|m_{0}|^2}{\rho_0}\;dx+\frac{1}{2}\int_{\O}\int_{\R^3}(1+|\v|^2)f_{0} \;dx d\v.
\end{split}
\end{equation*}
\end{Proposition}

\subsection{Pass the limit as $\delta\to0$}
The last step of showing the global weak solution is
to pass the limit as $\delta$ goes to zero. First, we let $(\rho^\delta, f^\delta,\u^{\delta})$ be a solution constructed by
Proposition \ref{P1}. It is easy to find that all estimates for $(\rho^\varepsilon,f^\varepsilon,\u^\varepsilon)$ still hold for $(\rho^\delta,f^\delta,\u^\delta)$, thus we can treat these terms as before.

It only remains to show the convergence of the terms $$\int_{\R^3}R_{\delta}\rho^\delta f^\delta(\u^\delta-\v)\,d\v,\,\,\,\text{ and }\,\,\, \Dv(R_{\delta}\rho^{\delta}(\u^{\delta}-\v)).$$

The next step is to deal with the convergence of
$\Dv_{\v}(R_{\delta}(\u^{\delta}-\v)\rho^{\delta}f^{\delta}).$
Let $\phi(\v)\in\mathcal {D}(\R^3)$ to be a test function, we
want to show
\begin{equation}\label{3.31}
\begin{split}
&\lim_{\delta\to 0}\left(\int_{\O}
R_{\delta}\rho^{\delta}\u^{\delta}\left(\int_{\R^3}
f^{\delta}\nabla_{\v}\phi d\v\right)-\int_{\O}\int_{\R^3}
R_{\delta}\rho^{\delta}f^{\delta}\v\nabla_{\v}\phi\right)
\\&=\int_{\O}\rho\u\left(\int_{\R^3}
f\nabla_{\v}\phi d\v\right)\,dx-\int_{\O}\int_{\R^3} \rho f\v\nabla_{\v}\phi\;d\v\,dx.
\end{split}
\end{equation}
To prove \eqref{3.31}, we introduce a new function $Q_{\delta}=1-R_{\delta}$ (see \cite{H}), it is easy to see that $$Q_{\delta}\to 0\;\;\text{ as } \delta\to 0.$$  Writing
\begin{equation}\label{3.32}
\begin{split}
\int_{\O} R_{\delta}\rho^{\delta}\u^{\delta}\left(\int_{\R^3}
f^{\delta}\nabla_{\v}\phi \,d\v\right)dx=&\int_{\O}
\rho^{\delta}\u^{\delta}\left(\int_{\R^3}
f^{\delta}\nabla_{\v}\phi\, d\v\right)dx
\\& -\int_{\O}
Q_{\delta}\rho^{\delta}\u^{\delta}\left(\int_{\R^3}
f^{\delta}\nabla_{\v}\phi\, d\v\right)dx.
\end{split}
\end{equation}
On one hand, applying Lemma \ref{l3+} to \eqref{3.4}, we   see that
\begin{equation}\label{3.33}
\int_{\R^3} f^{\delta}\nabla_{\v}\phi d\v\to\int_{\R^3} f\nabla_{\v}\phi d\v \;\; \text{ almost everywhere } (t,x).
\end{equation}
It is easy to see
\begin{equation}
\left|\int_{\R^3} f^{\delta}\nabla_{\v}\phi d\v\right|\leq C|m_{0}f^{\delta}|.
\end{equation}
This, combined with \eqref{3.14++}, strengthens  \eqref{3.33} as follows:
\begin{equation}\label{3.14+}
\int_{\R^3} f^{\delta}\nabla_{\v}\phi d\v\to\int_{\R^3} f\nabla_{\v}\phi d\v \;\; \text{ strongly in  } L^{\infty}(0,T;L^{2}(\O)).
\end{equation}
By the convergence of $\rho^{\delta}$, \eqref{3.14+} and the weak convergence of $\u^{\delta}$ in $L^{2}(0,T;H^1_0(\O))$, one deduces
\begin{equation*}
\int_{\O}
\rho^{\delta}\u^{\delta}\left(\int_{\R^3}
f^{\delta}\nabla_{\v}\phi \,d\v\right)dx\to \int_{\O}
\rho\u\left(\int_{\R^3}
f\nabla_{\v}\phi \,d\v\right)dx.
\end{equation*}
On the other hand,
\begin{equation}
\begin{split}
&\left|\int_{\O} Q_{\delta}\rho^{\delta}\u^{\delta}\left(\int_{\R^3}
f^{\delta}\nabla_{\v}\phi \;d\v\right)dx\right|
\\& \leq C\int_{\O} Q_{\delta}m_{0}f|\u^{\delta}|\;dxdt
\\& \leq C \|m_{0}f^{\delta}\|_{L^{\infty}(0,T;L^{2}(\O))}\|\u^{\delta}\|_{L^{2}(0,T;L^{6}(\O))}\|Q_{\delta}\|_{L^{2}(0,T;L^{3}(\O))}
\\ & \leq C\|Q_{\delta}\|_{L^{2}(0,T;L^{3}(\O))},
\end{split}
\end{equation}
which yields
\begin{equation*}
\left|\int_{\O} Q_{\delta}\rho^{\delta}\u^{\delta}\left(\int_{\R^3}
f^{\delta}\nabla_{\v}\phi d\v\right)\right|\to 0\;\;\text { as } \delta\to 0,
\end{equation*}
where we used that $m_{0}f^{\delta}$ is bounded in$ L^{\infty}(0,T;L^{2}(\O))$ and $\u^{\delta}$ is bounded in
$L^{2}(0,T;L^{6}(\O))$, and $Q_{\delta}\to 0$ strongly in $L^{2}(0,T;L^{3}(\O)).$

 So we have proved the convergence of the first integral on the left of \eqref{3.31}. We can treat  similarly the convergence of the second integral of \eqref{3.31}. Thus, we finish the proof of \eqref{3.31}.

To complete the proof of Theorem \ref{T1}, it only remains to check that $(\rho,\u,f)$ satisfies the energy inequality \eqref{2.3+}.
In order to verify the energy inequality \eqref{2.3+}, we need to show
\begin{equation}
\label{3.40}
\int_{0}^{t}\int_{\O}\int_{\R^3}R_{\delta}\rho^{\delta} f^{\delta}|\u^{\delta}-\v|^{2}\;dx d\v dt\to \int_{0}^{t}\int_{\O}\int_{\R^3}\rho f|\u-\v|^{2}\;dx d\v dt
\end{equation}
as $\delta \to 0.$

Denote \begin{equation*}
\begin{split}
&E^{\delta}=\int_{\O}\int_{\R^3} \rho^{\delta}f^{\delta}|\u^{\delta}-\v|^{2}\; d\v dx,
\\ & E^{\delta}=E_{1}^{\delta}-2E_{2}^{\delta}+E_{3}^{\delta},
\end{split}
\end{equation*}
where \begin{equation*}
E_{1}^{\delta}=\int_{\O}\int_{\R^3}\rho^{\delta}f^{\delta}|\u^{\delta}|^2\;d\v dx= \int_{\O}\rho^{\delta}m_{0}f^{\delta}|\u^{\delta}|^2\;dx,
\end{equation*}
\begin{equation*}
E_{2}^{\delta}=\int_{\O}\int_{\R^3}\rho^{\delta}f^{\delta}\u^{\delta}\v \;d\v dx =\int_{\O}\rho^{\delta} m_{1}f^{\delta}\u^{\delta}\;dx,
\end{equation*}
and  \begin{equation*}
E_{3}^{\delta}=\int_{\O}\int_{\R^3}\rho^{\delta}f^{\delta}|\v|^2\;d\v dx= \int_{\O}\rho^{\delta}m_{2}f^{\delta}\;dx .
\end{equation*}
Write $R_{\delta}E^{\delta}=E^{\delta}-Q_{\delta}E^{\delta},$ we consider the convergence of $E^{\delta}$ first.

Since
\begin{equation*}
\begin{split}
&\left|\int_0^T\int_{\O\times\R^3}\rho^{\delta}f^{\delta}|\u^{\delta}|^2\,d\v dx dt- \int_0^T\int_{\O\times\R^3}\rho f|\u|^2\,d\v dxdt\right|
\\&\leq \int_0^T\int_{\O}\left(\rho^{\delta}-\rho\right)m_0f^{\delta} |\u^{\delta}|^2\,dxdt+\int_0^T\int_{\O}\rho(m_0f^{\delta}-m_0f)|\u^{\delta}|^2\,dx\,dt
\\&+\int_0^T\int_{\O}\rho m_0f\left(|\u^{\delta}|^2-|\u|^2\right)\,dxdt,
\end{split}
\end{equation*}
then
\begin{equation*}
\int_{0}^{t}E_{1}^{\delta}\;dt\to \int_{0}^{t}\int_{\O}\int_{\R^3}\rho f|\u|^{2}d\v dx dt\quad
\quad\text{ as } \delta \to 0
\end{equation*}
for all $t>0$.
Similarly, we obtain
\begin{equation*}
\int_{0}^{t}E_{2}^{\delta}\;dt\to \int_{0}^{t}\int_{\O}\int_{\R^3}\rho f \u \v \;d\v dx dt\quad
\quad\text{ as } \delta \to 0
\end{equation*}
for all $t>0$.

Finally, let us examine
\begin{equation*}
\begin{split}
&\left|\int_0^t\int_{\O}\int_{\R^3}\rho^{\delta}f^{\delta}|\v|^2 \,d\v dxdt-\int_0^t\int_{\O}\int_{\R^3}\rho f|\v|^2 \,d\v dxdt\right|\\
&\leq \|\rho^{\delta}-\rho\|_{L^{\infty}}\int_0^T\int_{\O}m_2f^{\delta}dxdt+C\|\rho\|_{L^{\infty}}\int_0^t\int_{\O}(m_2f-m_2f^{\delta})dxdt
\\&=I_1+I_2.
\end{split}
\end{equation*}
It is clear that $I_1\to 0$ as $\delta\to0.$
For the term $I_2$,
because
 $$f^{\delta}\rightharpoonup f\;\;\text{ weak star in } L^{\infty}(0,T;L^{p}(\O\times\R^3))$$
 for all $p\in(1,\infty]$ and $m_3f^\delta$ is bounded in $L^{\infty}(0,T;L^1(\O))$, then for any fixed $r>0,$ we have
  \begin{equation*}
  \int_0^T\int_{\O\times\R^3}f^{\delta}|\v|^2\,dx d\v dt=\int_0^T\int_{\O\times\R^3}\chi(|\v|<r)|\v|^2f^{\delta}\,dx d\v dt+O(\frac{1}{r})
  \end{equation*}
  uniformly in $\delta$ where $\chi$ is the characteristic function of the ball of $\R^3$ of radius $r$. Letting $\delta\to 0,$ then $r\to\infty$, we find \begin{equation*}
  \int_0^T\int_{\O\times\R^3}f^\delta|\v|^2\,dx d\v dt\to   \int_0^T\int_{\O\times\R^3}f|\v|^2\,dx d\v dt,
  \end{equation*}
  which means $I_2\to 0$ as $\delta\to0.$
Thus, we have proved
\begin{equation*}
\int_{0}^{t}E^{\delta}dt\to\int_0^{t}\int_{\O}\int_{\R^3}\rho f|\u-\v|^2\;d\v dx dt\quad\text{ as } \delta \to 0.
\end{equation*}

In order to show \eqref{3.40}, it remains to show that
\begin{equation}
\label{global1+}
 \int_0^t Q_{\delta}E^{\delta}dt\to 0 \quad\text{ as } \delta\to 0.
\end{equation}
By the H\"{o}lder inequality, we obtain
\begin{equation*}
\begin{split}
&\int_{0}^{t}\int_{\O}Q_{\delta}\rho^{\delta}m_{0}f^{\delta}|\u^{\delta}|^{2}\;dxdt
\\&\leq
C\|Q_{\delta}\|_{L^{2}(0,T;L^{6}(\O))}\|m_{0}f^{\delta}
\|_{L^{\infty}(0,T;L^{2}(\O))}\|\u^{\delta}\|_{L^{2}(0,T;L^{6}(\O))}.
\end{split}
\end{equation*}
This, together with the definition of $Q_\delta$, implies that
$$\int_{0}^{t}\int_{\O}Q_{\delta}\rho^{\delta}m_{0}f^{\delta}|\u^{\delta}|^{2}dxdt\to 0\;\;\text{ as } \delta \to 0$$ for all $t>0$. Following the same argument, it is easy to see
\begin{equation*}
\int_{0}^{t}\int_{\O}Q_{\delta}\rho^{\delta}m_{1}f^{\delta}\u^{\delta} dxdt\to 0 \;\;\text{ as } \delta \to 0.
\end{equation*}
We write
\begin{equation*}
\begin{split}
&\int_{0}^{t}\int_{\O}\int_{\R^3}Q_{\delta}\rho^{\delta}|\v|^{2} f^{\delta}d\v dxdt=\int_{0}^{t}\int_{\O}\int_{|\v|\leq r}Q_{\delta}\rho^{\delta}|\v|^2f^{\delta}d\v dxdt+Q_{\delta}\frac{C}{r},
\\&=\int_{0}^{t}\int_{\O}\int_{\R^3}\chi({|\v|<r})Q_{\delta}\rho^{\delta}|\v|^2f^{\delta}d\v dxdt+\frac{C}{r}Q_{\delta}
\end{split}
\end{equation*}
uniformly in $\delta$,
where $\chi(x)$ is a characterized function. We have
\begin{equation*}
\rho^{\delta}\to \rho \;\;\text{ in } C([0,T];L^{p}(\O)) \text{ for any }1\leq p<\infty ,
\end{equation*}
and by the definition of $Q^{\delta}$, we have$$ \chi({|\v|<r})Q_{\delta}\rho^{\delta}\to 0\;\;\text{ strongly in }L^{p}(0,T;L^{q}(\O))\;\;\text{ for any }1\leq p,q <\infty.$$
It follows
\begin{equation*}
\int_{0}^{t}\int_{\O}\int_{\R^3}Q_{\delta}\rho^{\delta}|\v|^{2} f^{\delta}d\v dxdt\to 0
\end{equation*}
when letting $\delta \to 0$ and $r\to \infty.$
Thus, we have proved that \eqref{global1+},
and hence have proved \eqref{3.40}.


Thanks to the convergence facts and the convexity of the energy inequality, we deduce \eqref{2.3+} from energy inequality in Proposition \ref{P1}.

The proof of Theorem \ref{T1} is complete.

\bigskip\bigskip


\section*{Acknowledgments}
D. Wang's research was supported in part by the National Science
Foundation under Grant DMS-0906160 and by the Office of Naval
Research under Grant N00014-07-1-0668. C. Yu's research was supported in part by the National Science
Foundation under Grant DMS-0906160.

\end{document}